\documentclass[12pt]{article}

\usepackage[dvips]{graphicx}

\usepackage[bookmarks,colorlinks]{hyperref}

\usepackage{amsmath,amstext,amssymb,amsopn,amsthm}
\usepackage{amsmath,amssymb,amsthm}
\usepackage[mathscr]{eucal}

\usepackage{graphicx}
\usepackage{amssymb}
\usepackage{amsmath}
\usepackage{color}
\usepackage{times}

\newcommand{\abs}[1]{\left| #1 \right|}

\newcommand{\ex}{\mathbf{E}}
\newcommand{\pr}{\mathbf{P}}
\newcommand{\R}{\mathbb{R}}
\newcommand{\C}{\mathbb{C}}

\newcommand{\Rd}{{\R^d}}
\newcommand{\Rt}{{\R^2}}
\newcommand{\M}{{\mathcal{M}}}

\newcommand{\N}{{\mathcal{N}}}
\newcommand{\F}{{\mathcal{F}}}
\newcommand{\LL}{{L}}

\newcommand{\Lp}{{\LL^p}}

\newcommand{\Cinfc}{{C^\infty_c}}

\newcommand{\Lo}{{\LL^1}}

\newcommand{\Lt}{{\LL^2}}

\newcommand{\A}{{\mac{A}}}

\numberwithin{equation}{section}

\newcommand{\ind}{\mathbf{1}}
\newcommand{\wt}{\widetilde}
\newcommand{\pl}{\|}

\newcommand {\mac}[1] { \mathbb{#1} }

\newcommand{\sphere}{{\mathbb{S}}}
\newcommand{\Beta}{\mathcal{B}}

\newtheorem{theorem}{Theorem}[section]
\newtheorem{lemma}{Lemma}[section]

\begin{document}
\sloppy
\thispagestyle{empty}

 
\title{Fourier multipliers for non-symmetric L\'evy processes}
\author{
Rodrigo Ba\~nuelos\thanks{Supported in part by NSF grant \#0603701-DMS.},
 Adam Bielaszewski\thanks{Supported by grant KBN 1 P03A 026 29.},
 Krzysztof Bogdan\thanks{Supported in part by KBN 1 P03A 026 29 and MNiSW N N201 397137.}}
\date{\today}
\maketitle


\begin{abstract}
{\it We study Fourier multipliers resulting from martingale transforms of general L\'evy processes.}
\end{abstract}
\footnotetext{2000 {\it MS Classification}: 42B15, 60G51 (Primary), 60G46, 42B20 (Secondary).}


\section{Introduction}

For each bounded function $M:\,\Rd\to \C$ there is a unique bounded
linear operator $\M$ on $\Lt(\Rd) $ defined in terms of the Fourier
transform as follows,
\begin{equation}
  \label{eq:dmf}
\widehat{\M g}=M\hat g\,.
\end{equation}
The operator norm of $\M$ on $\Lt(\Rd) $ is $\|\M\|=\|M\|_\infty$.
It has long been of interest to study {\it symbols} $M$ for which the {\it Fourier multiplier} $\M$ extends to a bounded linear
operator on $L^p(\Rd)$ for $p\in (1,\infty)$.
Fourier multipliers resulting from transforming jumps of symmetric L\'evy
process have been recently obtained in \cite{MR2345912}.
By using Burkholder's inequalities for differentially subordinate
continuous time martingales with jumps \cite{Burkholder} in the general form of Wang \cite{MR1334160}, 
we proved 
that their operator norms on $L^p(\Rd)$ do not exceed  
\begin{equation}
  \label{eq:psm1}
p^*-1=\max\{p-1, \frac{1}{p-1}\}\,.
\end{equation}
For a broad discussion of Burkholder's method and its many extensions and applications, we refer the reader to \cite{2010RB}.
{In this note
we adapt the methods of \cite{MR2345912} to non-symmetric L\'evy 
processes.} The resulting multipliers are given
by (\ref{eq:fs}) and Theorem~\ref{th:bfm} below.
We remark that for $\mu = 0$ and symmetric $V$
the result was proved in \cite[Theorem~1]{MR2345912}.
The present Theorem~\ref{th:bfm} is a generalization, but the symbols (\ref{eq:fs}) are very similar to those given in \cite{MR2345912}.

Given a Borel measure $V\geq 0$ 
on $\Rd$ such that $ V(\{0\})=0$ and
\begin{equation}
  \label{eq:clm}
\int_\Rd \min(|z|^2,1) V(dz)<\infty
\end{equation}
(that is, a L\'evy measure), a finite Borel measure $\mu\geq
0$ on the unit sphere $\sphere$ in $\Rd$, and Borel measurable complex-valued functions $\phi$ on
$\Rd$ and $\varphi$ on $\sphere$ such that $\|\phi\|_\infty\leq 1$ and $\|\varphi\|_\infty\leq 1$, we define
\begin{equation} \label{eq:fs}
M\left(\xi\right)=
\frac{
        \int_\Rd \left[1 - \cos (\xi  ,  z) \right]\phi\left(z\right)V\left(dz\right) 
        +
            \frac{1}{2} \int_{\sphere} \left( \xi ,  \theta \right)^2 \varphi\left( \theta \right) \mu\left(d\theta \right)
          }
     {
     \int_\Rd \left[1 - \cos(\xi  ,  z) \right]V\left(dz\right) 
     +
     \frac{1}{2}\int_{\sphere} \left( \xi ,  \theta \right)^2 \mu\left(d\theta \right)
     } \ , 
\end{equation} 
where we let $M(\xi)=0$, if the denominator equals zero. Clearly, $\|M\|_\infty\leq 1$. 
Here and for the rest of this paper, the pairing between vectors,
\begin{equation}\label{eq:dsp}
(\xi,\eta)=\sum_{n=1}^d \xi_n\eta_n \ , \quad \mbox{ if } \xi,\eta \in \R^d \mbox{ or } \C^d \ ,
\end{equation}
is without complex conjugation. We also denote $|\xi|^2=\sum_{n=1}^d |\xi_n|^2=(\xi,\overline{\xi})$.
If the denominator in (\ref{eq:fs}) vanishes on a set of positive Lebesgue measure, then $V=0$ (see \cite[Section 3]{MR2345912}), hence $\mu=0$ and  $M\equiv 0$.

\begin{theorem} \label{th:bfm}
If $1<p<\infty$ and $\M$ is defined by
{\rm (\ref{eq:dmf})} and {\rm (\ref{eq:fs})}, then
\begin{equation}
  \label{eq:dM}
\|\M g\|_p\leq (p^*-1)\|g\|_p\,,\quad g\in L^p(\Rd)\,.
\end{equation}
\end{theorem}
In particular,  letting $V=0$ in (\ref{eq:fs}) yields the symbol
\begin{equation} \label{BMeq:fs}
M\left(\xi\right)=
\frac{
     \int_{\sphere} \left( \xi ,  \theta \right)^2 \varphi\left( \theta \right) \mu\left(d\theta \right) }
     {
     \int_{\sphere} \left( \xi ,  \theta \right)^2 \mu\left(d\theta \right)
    }\,,
\end{equation}
or 
\begin{equation}\label{eq:pc}
M(\xi)=\frac{(\mac{A}\xi,\xi)}{(\mac{B}\xi,\xi)}\,,
\end{equation}
where $\mac{A}=\left[\mac{A}_{k,l}\right]_{k,l=1, \ldots, d}$ and $\mac{B}=\left[\mac{B}_{k,l}\right]_{k,l=1, \ldots,d}$ are given by
\begin{eqnarray} \label{eq:mx}
 \mac{A}_{k,l} = \int_{\sphere}
   \theta_{k}\theta_{l} \ \varphi\left(\theta\right)\mu(d\theta) \ ,
   \quad 
 \mac{B}_{k,l} =  \int_{\sphere} \theta_{k}\theta_{l} \ \mu(d\theta)  \ .
\end{eqnarray}
For instance, the approach  yields  the bound $p^*-1$ for the multiplier with the symbol $-2\xi_1\xi_2/|\xi|^2$, via
\begin{equation}\label{eq:R2}
\mac{A}=
\left[ \begin{array}{ccccc}
0 & -1 & 0 & \cdots & 0 \\
-1 & 0 & 0 & \cdots & 0  \\
\vdots & \vdots & \vdots & \ddots & \vdots \\
0 & 0 & 0 & \cdots & 0 \end{array} \right] \ , \quad
\end{equation}
and $\mac{B}=\mac{I}$, the identity matrix (\cite{banuelos-mendez}). In this way we obtain $2R_1R_2$, a second order Riesz transform multiplied by two, see Section~\ref{s:d}.
It is known that  the norm of the operator actually equals $p^*-1$ (\cite[Corollary 3.2]{MR2551497}), and so the constant in (\ref{eq:dM}) cannot be improved in general. 
Our method will also give
the upper bound $2(p^*-1)$ for the norm of the multiplier with the symbol $(\xi_1-i\xi_2)^2/|\xi|^2$, via
\begin{equation}\label{eq:BAm}
\mac{A}=
\left[ \begin{array}{cc}
1 & -i \\
-i & -1
\end{array} \right] \quad \mbox{and}\quad \mac{B}=\mac{I}\ ,
\end{equation}
see Section~\ref{BMs:p}. In this connection we remark that $|\mac{A}\xi|=\sqrt{2}|\xi|$ for $\xi \in \R^2$ and $|\mac{A}\xi|\leq 2|\xi|$ for $\xi \in \C^2$. The multiplier is called the Beurling-Ahlfors transform, and its norm is actually smaller than $2(p^*-1)$, see \cite{2010RB}. 
In particular, the celebrated conjecture of T. Iwaniec asserts that the norm equals precisely $p^*-1$. There is some evidence, given by Lemma~\ref{lem:c2} below that our approach cannot improve the bound $2(p^*-1)$.

The paper is organized as follows. 
Section~\ref{BMs:p} has a didactic purpose. We namely consider  $\mac{B}=\mac{I}$ in (\ref{eq:mx}). 
This case can be resolved 
by means of the standard
It\^o calculus for the Brownian motion.   
This argument was first given in \cite{banuelos-mendez},  and has since appeared in many different places and settings, but we believe it is worth repeating here
with notation 
emphasizing analogies with Section~\ref{s:p}. In this way we hope to make the rest of the paper more readable for those less familiar with the stochastic calculus of L\'evy processes. 
In Section~\ref{s:p} we give the proof and a discussion of Theorem~\ref{th:bfm}. First of all, by using a  simple algebra we reduce the symbols (\ref{eq:fs}) to those of \cite[Theorem~1]{MR2345912}. This gives a proof but not much insight, since \cite{MR2345912} only concerns symmetric L\'evy processes. Therefore in the remainder of Section~\ref{s:p} we present stochastic calculus leading to  the symbols (\ref{eq:fs}). Our main purpose is to explain why non-symmetry of the process is not reflected in the symbol. For instance we will see in (\ref{eq:znikabe}) that the drift of the L\'evy process does not contribute to $M$.
Examples  and further discussion are given in Section~\ref{s:d}.

Throughout the paper the functions, measures and sets in $\Rd$ will be assumed Borelian. 
For $1\leq p<\infty$ we denote by $\Lp=\Lp(\Rd)$ the family of all the complex-valued functions $f$ on $\Rd$ which have finite norm
$
\pl g\pl_p =\left[\int_\Rd |f(x)|^p dx\right]^{1/p}
$. As usual, we will identify functions equal almost everywhere. We also denote
$\pl f\pl_\infty ={\mbox{ess sup}}_{x\in \Rd}|f(x)|$, and we let
$\Cinfc$ be the class of all the smooth compactly supported numerical functions on $\Rd$.
We recall that $\Cinfc$ is 
dense in $\Lp$  for  each $p\in [1,\infty)$.
Our convention for the Fourier transform will be 
\begin{displaymath}
  \widehat{f}(\xi)=\int_\Rd e^{i (\xi, x)}f(x)dx\,,\quad \xi\in \Rd\,.
\end{displaymath}
If $\rho$ is a probability measure on $\Rd$ and $k\in \Lo$, then Fubini's theorem yields
\begin{equation}
  \label{eq:wrtlm}
\int_\Rd\int_\Rd k(x+y)\rho(dy)dx=\int k(x)dx\,.
\end{equation}

\section{Brownian martingales and It\^o calculus}\label{BMs:p}
In this section we present a simple approach to Fourier multipliers with symbols of the form
(\ref{eq:pc}).
We will use the familiar It\^o calculus for the Brownian motion, for which we refer the reader to \cite{MR2020294}, \cite{MR1303781}  or \cite{MR1780932}. 
The main ideas will be similar to those in Section~\ref{s:p} below, but the calculations are shorter and simpler. As already mentioned, we hope that this part of the paper will be easier to read for those familiar with the basics of the It\^o calculus but perhaps not as familiar with the stochastic calculus of jump processes used in Section~\ref{s:p}.

We let $\pr$ and $\ex$ be the probability and expectation for a Brownian motion $\left(B_{t},\,t\geq 0\right)$ on $\Rd$. 
We will consider the filtration
\begin{displaymath}
  \F_t=\sigma\{B_{s}\,;\; 0\leq s\leq t\}
\,,\quad t\geq 0\,,
\end{displaymath}
and the Gaussian kernel
\begin{equation}
  \label{BMeq:dp}
p_{t}(x)=(2\pi t)^{-d/2} \exp(-|x|^2/2t)\,,\quad t>0,\;x\in \Rd\,.
\end{equation}
It is well-known that $p_{t}(x)dx$ is the distribution of $B_{t}$ for $t>0$, $p_s*p_t=p_{s+t}$, 
\begin{equation}
\label{BMeq:lkf}
\widehat{p_t}(\xi)
=e^{-t|\xi|^2/2}\,, \quad \xi \in \Rd\,,
\end{equation}
and that the heat equation holds for $p_t(x)$,
\begin{equation}\label{BMeq:pe}
\frac{\partial}{\partial t}p_t(x)=\frac{1}{2}\Delta p_t(x)\,.
\end{equation}
In what follows, $f,g\in \Cinfc$, $x\in \Rd$, and $0\leq t \leq u<\infty$. We denote $p_{0}*f(x)=f(x)$, and we have
\begin{equation}\label{BMeq:nacc}
  \ex f(x+B_{t})=p_{t}*f(x)\,.
\end{equation}
We consider the following Brownian {\it parabolic martingale},
\begin{equation}
  \label{BMeq:dpm}
  F_t=F_t(x;u,f)=\ex\left(f(x+B_u)|\F_t\right)=p_{u-t}*f(x+B_{t})\,.
\end{equation}
Regardless of $t$, the entire time interval $[0,u]$ is involved in $F_t$. Indeed, the ''evolution" from $0$ to $t$ proceeds via the Brownian motion, while that from $t$ to $u$ goes by its expectations.
In fact, the martingale equals an It\^o integral plus a constant, as we verify by applying It\^o formula and (\ref{BMeq:pe}) to the function $(t,y)\mapsto p_{u-t}*f(x+y)$,
\begin{eqnarray}
\nonumber
&&F_t-F_0=\int_0^t \left(\frac{\partial}{\partial v}p_{u-v}*f\right)(x+B_{v})dv
+
\int_0^t \nabla p_{u-v}*f(x+B_{v})dB_v\\
&&+\int_0^t \frac{1}{2}\Delta p_{u-v}*f(x+B_{v})dv
=
\int_0^t\nabla p_{u-v}*f(x+B_{v})dB_v\,.
\label{eq:Ii}
\end{eqnarray}
$F$ is bounded, hence square integrable.
The quadratic variation of $F$ is
\begin{equation}\label{eq:qvGc}
[F,F]_t=|F_0|^2+\int_0^t\left|\nabla p_{u-v}*f(x+B_{v})\right|^2dv\,.
\end{equation}
Let $\mac{A}$ be a real or complex $d\times d$ matrix  such that
\begin{equation}\label{sub1}
|\mac{A}z|\leq |z|\,, \quad z \in \C^d\,.
\end{equation}
We also consider the martingale 
\begin{equation}\label{eq:mtA}
G_t=G_t(x;u,g,\mac{A})=\int_0^t \mac{A}\nabla p_{u-v} g(x+B_{v})dB_v. 
\end{equation}
The quadratic variation of $G$ is 
\begin{equation}
\label{eq:qvF}
[G,G]_t=\int_0^t\left|\mac{A} \nabla p_{u-v} g(x+B_{v})\right|^2dv\,.
\end{equation} 
By (\ref{eq:qvGc}), (\ref{sub1}) and (\ref{eq:qvF}),
$G=G(x;u,g,\mac{A})$ is {\it differentially subordinate} to $H=F(x;u,g)$,  
in the following sense introduced in \cite{banuelos-wang}:
\begin{equation}\label{eq:dsu}
\!\!\!\!\!0\leq [H,H]_t-[G,G]_t \quad \mbox{is non-decreasing in}\; t\,.
\end{equation}
Let $1<p<\infty$ and $p^*=\max\{p-1, (p-1)^{-1}\}$.
By \cite[Theorem~2]{banuelos-wang},
\begin{equation}\label{BMeq:wmt}
   \ex |G_t(x;u,g,\mac{A})| ^p\leq (p^*-1)^p   \,\ex |F_t(x;u,g)|^p\,.
\end{equation}
Let $t=u$. We have $F_u(x;u,f)=f(x+X_{u})$. By (\ref{BMeq:wmt}) and (\ref{eq:wrtlm}),
\begin{eqnarray} 
 \nonumber
\int_\Rd \ex  
|G_u(x;u,g,\mac{A})|^p dx 
&\leq& 
(p^*-1)^p   
\int_\Rd \ex  
|F_u(x;u,g)|^p dx \\
\label{BMeq:psm}
&=&(p^*-1)^p \pl g \pl_p^p\,.
\end{eqnarray}
Let $q=p/(p-1)$. By H\"older's  inequality for $\pr \otimes dx$, (\ref{BMeq:psm}) and (\ref{eq:wrtlm}),
\begin{equation}
\label{MBeq:bof}
\int_\Rd \ex |G_u(x;u,g,\mac{A}) f(x+B_{u})| dx \leq 
(p^*-1)\pl g \pl_p
\pl f \pl_q\,.
\end{equation}
Using (\ref{eq:qvGc}) and (\ref{eq:qvF}) we obtain
\begin{align}
&\ex G_u(x;u,g,\mac{A}) {F_u(x;u,f)}=\ex G_u (F_u-F_0)\nonumber\\
&   = 
\ex \!\!\int_0^u \big(\mac{A}\nabla p_{u-v}*g(x+B_{v}),\nabla p_{u-v}*f(x+B_{v})\big)dv \nonumber \\
&  = \int_0^u \int_\Rd \big(\mac{A}\nabla p_{u-v}*g(x+y),\nabla p_{u-v}*f(x+y)\big)p_v(dy)dv\,.
\label{eq:rsp}
\end{align}
In view of (\ref{MBeq:bof}) we consider
$$
\Lambda(g,f)=\int_\Rd \ex G_u(x;u,g,\mac{A}) F_u(x;u,f) dx \ .
$$
By (\ref{eq:rsp}), (\ref{eq:wrtlm}) and the properties of the Fourier transform we obtain
\begin{align}
\Lambda(g,f)&=\nonumber
\int_0^u\int_\Rd 
\big(\mac{A}\nabla p_{u-v}*g(x), \nabla p_{u-v}*f(x)\big)dxdv\\ \nonumber
&=(2\pi)^{-d}\int_0^u\int_\Rd 
\big(\mac{A}\xi,\xi\big) \hat p_{v}^2(\xi) \hat g(\xi) {\hat f(-\xi)}d\xi dv\\
&=\nonumber
(2\pi)^{-d}\int_\Rd 
(1-e^{-u|\xi|^2})\big(\mac{A}\xi,\xi\big)|\xi|^{-2} \hat g(\xi) {\hat f(-\xi)}d\xi\ .
& \nonumber
\end{align}
The above use of Fubini's theorem is justified since
\begin{align*}
&\int_\Rd \int_0^u\int_\Rd  |p_{u-v}*\nabla g(x+y)|\  |p_{u-v}*\nabla f(x+y)| \,
p_{v}(dy) dv dx \nonumber\\
&\le u \pl \nabla f \pl_\infty \pl \nabla g\pl_1 
<\infty \ .
\end{align*}
By (\ref{MBeq:bof}) we have $|\Lambda(g,f)|\leq (p^*-1)\pl g \pl_p
\pl f \pl_q$. If $f$ is fixed, then by the Riesz representation theorem there is a function $h$ such that $\pl h\pl_p\leq (p^*-1)\pl g\pl_p$ for all $p\in (1,\infty)$. In particular, $h\in \Lt$, and
$$
\Lambda(g,f)= \int_\Rd h(x){f}(x)dx\\
= (2\pi)^{-d} \int_\Rd \hat h(\xi){\hat f(-\xi)}d\xi\,,
$$
hence $h=\M_u g$, where $\M_u$ is the Fourier multiplier with the symbol $(1-e^{-u|\xi|^2})\big(\mac{A}\xi,\xi\big)/|\xi|^{2}$. 
Let $\M$ be the Fourier multiplier with the symbol $M(\xi)=\big(\mac{A}\xi,\xi\big)/|\xi|^{2}$. We let $u\to \infty$.  By Plancherel's theorem and bounded convergence of the symbols, $\M_u g\to \M g$ in $\Lt $.
There is a sequence $u_n\to \infty$ such that $\M_{u_n}g\to \M g$ almost everywhere. 
Fatou's lemma yields $\pl \M g\pl_p\leq (p^*-1)\pl g\pl_p$, and  
$\M$ extends uniquely to the
whole of $\Lp$ without increasing
the norm. We conclude that the Fourier multiplier with the symbol $\big(\mac{A}\xi,\xi\big)/|\xi|^{2}$ has the norm at most $p^*-1$ on $\Lp$ for $1<p<\infty$, provided 
(\ref{sub1}) holds.

If $\mac{A}\neq 0$ is a general square real or complex $d\times d$ matrix, then $\mac{A}/\pl \mac{A}\pl$ satisfies (\ref{sub1}), hence 
 the Fourier multiplier with the symbol $(\mac{A}\xi,\xi)/|\xi|^2$ has the norm bounded by $\pl \mac{A}\pl (p^*-1)$ on $\Lp$. Here $\pl \mac{A} \pl$ is the (spectral) operator norm of $\mac{A}$, induced by the Euclidean norm on $\C^d$. On occasions, 
 $\nabla p_{u-v}*g(x)$ will have a restricted range of values, and then the inequality in (\ref{sub1}) needs only to hold in this range. In particular, the Beurling-Ahlfors transform given by (\ref{eq:pc}) and (\ref{eq:BAm}) has the norm at most
$2(p^*-1)$ when acting on complex-valued functions, and at most $\sqrt{2}(p^*-1)$ when restricted to real-valued functions, see also Section~\ref{s:d}.

The above calculations of the symbol reflect the identity
$(\mac{A}\xi,\xi)/|\xi|^2=(\mac{A}\xi,\xi)\int_0^\infty \exp (-2t|\xi|^2/2)dt$.
A {semigroup} interpretation of similar calculations is proposed in \cite[(36) and (37)]{MR2345912}. 
As already mentioned in the Introduction, the symbols (\ref{eq:pc}) and their $\Lp$ estimates are not new. We refer the reader to \cite{2010RB} for a detailed discussion of  further symbols that can be obtained by transformations of more general It\^o integrals, and for their applications.

\section{L\'evy-It\^o calculus and Fourier multipliers}\label{s:p}
\begin{proof}[Proof of Theorem~\ref{th:bfm}]

We will first consider $\mu = 0$ in (\ref{eq:fs}), i.e. we will prove the theorem for symbols of the form
\begin{equation}\label{eq:fs_bez_mu}
\frac{
      \int \left[1 - \cos (\xi  ,  z) \right]\phi\left(z\right)V\left(dz\right)
     }  
     {
      \int \left[1 - \cos (\xi  ,  z) \right]V\left(dz\right)
     }.
\end{equation}
For $A\subset \Rd$ we let $\breve{V}(A)=[V(A)+V(-A)]/2$
(the symmetrization of $V$), $\tilde{V}(A)=[V(A)-V(-A)]/2$ (the antisymmetric part of $V$). We also define $\breve{\phi}(z)=[\phi(z)+\phi(-z)]/2$,
$\tilde{\phi}(z)=[\phi(z)-\phi(-z)]/2$ for $z\in \Rd$.
The function $z\mapsto\cos(\xi, z)$ is symmetric, hence  
$\int_\Rd [1-\cos (\xi ,  z)] V(dz)=\int_\Rd [1-\cos( \xi , 
z)] \breve{V}(dz)$. We note that 
$$
\phi V=(\breve{\phi}+\tilde{\phi})(\breve{V}+\tilde{V})=
\big(\breve{\phi}\breve{V}+\tilde{\phi}\tilde{V}\big)
+\big(\breve{\phi}\tilde{V}+\tilde{\phi}\breve{V}\big)
$$
as measures, and so for every $\xi \in \Rd$ we have
  \begin{equation}
    \label{eq:fs3}
\frac{\int_\Rd[1-\cos(\xi ,  z)]\phi(z)V(dz)}
{\int_\Rd [1-\cos (\xi ,  z)] V(dz)}
=\frac{\int_\Rd[1-\cos(\xi ,  z)]
\big(\breve{\phi} \breve{V}+\tilde{\phi} \tilde{V}\big)(dz)}
{\int_\Rd [1-\cos (\xi ,  z)] \breve{V}(dz)}\,.
  \end{equation}
Since $\breve{V}+\tilde{V}=V\geq 0$, we have that
$\tilde{V}=k\breve{V}$, with an antisymmetric real function $k$ such that $|k|\leq 1$.
Thus, in the numerator of (\ref{eq:fs3}) we integrate against
$\phi^*\breve{V}$, where
$\phi^*=\breve{\phi}+k\tilde{\phi}=
\frac{1+k}{2}(\breve{\phi}+\tilde{\phi})
+\frac{1-k}{2}(\breve{\phi}-\tilde{\phi})$, a convex combination.
If $|\phi|\leq 1$ on $\Rd$ then $|\breve{\phi}\pm \tilde\phi|\leq 1$
on $\Rd$. By convexity we see that $|\phi^*|\leq 1$.
Application of \cite[Theorem~1]{MR2345912} to $\breve{V}$ and $\phi^*$ gives the $\Lp$ estimate (\ref{eq:dM})  for the Fourier multiplier with the symbol (\ref{eq:fs_bez_mu}).

We will now prove the general result.
Consider $M$ given by (\ref{eq:fs}) and let $\varepsilon>0$.
In polar coordinates $(r,\theta) \in (0,\infty)\times\sphere$
we define the L\'evy measure
\begin{equation} \nonumber
  \nu_{\varepsilon}(drd\theta) = \varepsilon^{-2}\delta_{\varepsilon}(dr)\mu(d\theta) \,.
\end{equation}
Here $\delta_\varepsilon$ is the probability measure concentrated on $\{\varepsilon\}$.
We consider the multiplier $\M_\varepsilon$ on $\Lt $ with the symbol $M_\varepsilon$
defined by (\ref{eq:fs_bez_mu}),
where the L\'evy measure is replaced by $\ind_{\{ \abs{z} > \varepsilon \}}V + \nu_\varepsilon$
 and the jump modulator is replaced by 
$\ind_{\{ \abs{z} > \varepsilon \}}\phi(z) + \ind_{\{ \abs{z} = \varepsilon \}}\varphi(z/|z|)$.
We let $\varepsilon\to 0$ and note that
\begin{eqnarray}
\int_\Rd [1-\cos(\xi, z)]\varphi(z/|z|) \nu_\varepsilon\left(dz\right)
&=& \int_{\sphere} (\xi  ,  \theta)^2 \varphi(\theta)
 \frac{[1-\cos(\xi,\varepsilon\theta)]}{(\xi  ,  \varepsilon\theta)^2}
 \mu(d\theta) \nonumber \\
 &\to&
 \frac{1}{2}\int_{\sphere} (\xi  ,  \theta)^2 \varphi(\theta) \label{eq:lcs}
 \mu(d\theta)
 \,,
\end{eqnarray}
therefore $M_\varepsilon\to M$, where $M$ is given by (\ref{eq:fs}).
Let $1<p<\infty$ and $g\in \Lt \cap \Lp$.
By Plancherel's theorem and bounded pointwise convergence of the symbols, $\M_\varepsilon g\to \M g$ in $\Lt$ as $\varepsilon \to 0$.
There is a sequence $\varepsilon_n\to 0$, such that $\M_{\varepsilon_n}g\to \M g$ almost everywhere. By Fatou's lemma and the conclusion of the first
part of the proof applied to $\M_{\varepsilon_n}$ we have that $\pl \M g\pl_p\leq (p^*-1)\pl g\pl_p$. 
\end{proof}

In the remainder of this section we 
will show how the symbol in (\ref{eq:fs_bez_mu}) results
from transforming martingales related to non-symmetric L\'evy processes.
Our main purpose is to elucidate as clearly as possible at which point the drift and asymmetry of the L\'evy measure disappear from the picture, so that only symmetric symbols  (\ref{eq:fs_bez_mu}) obtain. The phenomenon was quite a surprise to the authors and may be important
in extending the methods of this paper. 
We will closely follow the development of \cite{MR2345912}, but the presentation is  simpler than that in \cite{MR2345912}, and essentially self-contained.
The reader may also consult
\cite{MR2020294} or \cite{MR1780932} 
for general information about the stochastic calculus of jump processes.

For a measure $\mu$, set $A$, function $f$, and point $a$,  we define 
$\check{\mu}(A)=\mu(-A)$, $\mu(f)=\int f(x)\mu(dx)$, $(f\mu)(A)=\int_A f(x)\mu(dx)$, 
$f^a(x)=f(x+a)$, and $(\mu)^a(f)=\int f(x+a)\mu(dx)=\mu(f^a)$.

Let $\nu\geq 0$ be an arbitrary {\it finite} nonzero measure on $\Rd$ not
charging the origin. Let $|\nu|=\nu(\Rd)$ and $\wt{\nu}=\nu/|\nu|$. 
Let $\pr$ and $\ex$ be the probability and expectation for a family of
independent random variables $T_i$ and $Z_i$, $i=1,2,\ldots$, where each $T_i$ is
exponentially distributed with $\ex T_i=1/|\nu|$, and each $Z_i$ has
$\wt{\nu}$ as its distribution.
We let $S_i=T_1+\ldots+T_i$, for $i=1,2,\ldots$.
For $0\leq t<\infty$ we let $X_{t}=\sum_{ S_i\leq t} Z_i$,
$X_{t-}=\sum_{ S_i< t} Z_i$ and $\Delta X_{t}=X_{t}-X_{t-}$.
We note that $\N(B)=\#\{i:\;(S_i,Z_i)\in B\}$ is a Poisson random measure on
$(0,\infty)\times \Rd$ with the intensity measure $dv\,\nu(dx)$, and
$X_{t}=\int_{ v\leq t} x \N(dvdx)$ is the L\'evy-It\^o
decomposition of $X$; see  \cite{MR1739520}.
Alternatively, we may consider $\N$ as the initial datum, and then $(S_i,Z_i)$ may be defined as the atoms of $\N$.
The number of {\it signals} $S_i$ such that $S_i\leq t$ 
equals $N(t)=\N((0,t]\times \Rd)$, a random variable with Poisson distribution of parameter $|\nu|t$.
We will consider  the generic compound Poisson process with the drift,
\begin{equation}
  \label{eq:dcPp}
X_{t}^b=X_{t}+tb\,.
\end{equation}
Here $b\in \Rd$. 
It is well-known that every L\'evy process
on $\Rd$ can be obtained as a limit of such processes. Again, we refer the reader to  \cite{MR1739520}.
As we will see, the study of $\{X^b_{t}\}$ easily
reduces to that of $\{X_{t}\}$, or to the case of $b=0$. For instance, our notation gives
\begin{equation}
  \label{eq:dcPpf}
\ex f(X_{t}^b)=\ex f^{tb}(X_{t})\,.
\end{equation}
\begin{lemma} \label{lem:ls}
If $F:\; \R\times \Rd\times \Rd\to\C$ is bounded and $0\leq t<\infty$, then
  \begin{equation}
    \label{eq:els}
    \ex \sum_{ S_i\leq t} F(S_i,X^b_{S_{i}-}, X^b_{S_i})
= \ex \int_0^t \int_\Rd F(v,X^b_{v-},
X^b_{v{-}}+z)\nu(dz)dv\,.
  \end{equation}
\end{lemma}
\begin{proof}
By considering $F^*(v,x,y)=F(v,x+vb,y+vb)$ we may assume that
$b=0$ in (\ref{eq:els}).
In this case the proof of \cite[Lemma~1]{MR2345912}
applies (the symmetry of $\nu$ was not used in that proof).
For clarity we note that $N(t)$ is exponentially integrable, and so is the sum in  (\ref{eq:els}).
\end{proof}

In particular, for finite $t\geq 0$ and bounded $F$ we have
\begin{eqnarray}
&& \ex \sum_{ S_i\leq t}\left[ F(S_i,X^b_{S_{i}-},
  X^b_{S_i})-F(S_i,X^b_{S_{i}-}, X^b_{S_{i}-})\right]\nonumber \\
&& = \ex \int_0^t \int_\Rd \left[F(v,X^b_{v-}, X^b_{v-}+z)-F(v,X^b_{v-},X^b_{v-})\right]\nu(dz)dv\,.
\quad  \label{eq:d-c}
\end{eqnarray}
In what follows we will consider the filtration
\begin{displaymath}
  \F_t=\sigma\{X_{t}\,;\; 0\leq s\leq t\}=\sigma\{X^b_{t}\,;\; s\leq t\}
\,,\quad t\geq 0\ .
\end{displaymath}
For $t\geq 0$ we define 
\begin{equation}
  \label{eq:dp}
p_{t}=e^{-t|\nu|}\sum_{n=0}^\infty
\frac{t^n}{n!}\nu^{*n}=e^{*t(\nu-|\nu|\delta_0)}
=\sum_{n=0}^\infty \frac{t^n}{n!}(\nu-|\nu|\delta_0)^{*n}
\,.
\end{equation}
The series converges in the norm of absolute variation of measures.
Clearly,
\begin{equation}
  \label{eq:dpt}
  \frac{\partial}{\partial t}p_t=(\nu-|\nu|\delta_0)*p_t\,,
\end{equation}
and $p_{s}*p_{t}=p_{s+t}$ for $s,t\geq 0$. 
By (\ref{eq:dp}), $p_{t}$ is the distribution of $X_{t}$ (as well as of $X_{t-}$), and the sides of (\ref{eq:els}) equal 
\begin{equation}\label{eq:essem}
\int_0^t\int_\Rd  \int_\Rd F(v,y+vb,y+vb+z)\nu(dz)p_{v}(dy)dv\,.
\end{equation}
Let 
\begin{equation}
  \label{eq:lke}
  \Psi(\xi)=\int_\Rd \left[e^{i(\xi,z)}-1\right]\nu(dz)\,,\quad \xi \in \Rd\,.
\end{equation}
We directly verify that $\Psi$ is bounded and continuous on
$\Rd$, $\Psi(-\xi)=\overline{\Psi(\xi)}$,
$\Re\Psi(\xi) = \int_\Rd [\cos (\xi, z) - 1] \nu(dz)$ (compare the denominator in (\ref{eq:fs_bez_mu})), and
\begin{equation}
  \label{eq:lkf}
\widehat{p_t}(\xi)=\int_\Rd e^{i(\xi,x)}p_t(dx)
=e^{t\Psi(\xi)}\,, \quad \xi \in \Rd\,.
\end{equation}
$\Psi$ is the L\'evy-Khinchine exponent and (\ref{eq:lkf}) is the L\'evy-Khinchin formula for $X$.
We also consider the convolution semigroup with the drift vector $b$,
$$p^b_t=(p_t)^{tb}\,, \quad t\geq 0\,,$$ 
that is
$p^{b}_t(f)=p_t(f^{tb})$. We have
\begin{equation}
  \label{eq:lkfb}
\widehat{p^{b}_t}(\xi)=\int_\Rd e^{i(\xi,x+tb)}p_t(dx)
=e^{it(\xi,b)+t\Psi(\xi)}\,, \quad \xi \in \Rd\,.
\end{equation} 
In what follows we let
$f,g \in \Cinfc$, $x\in \Rd$ and $0\leq t\leq u<\infty$. We define
\begin{equation}\label{eq:nacc}
  P^b_{t} g(x)=\ex g(x+X^b_{t})=\int_\Rd g(x+y)p^b_{t}(dy)\,.
\end{equation}
This is the convolution with  the reflection of $p^b_{t}$, and we have
\begin{equation}
  \label{eq:nac}
  \widehat{P^b_{t} g}(\xi)=\hat{g}(\xi)\widehat{p^{b}_{t}}(-\xi)
=\hat{g}(\xi)e^{-it(\xi,b)+t\Psi(-\xi)}\,, \quad \xi \in \Rd\,.
\end{equation}
We denote $P_{t}=P^0_{t}$. By (\ref{eq:nacc}) we have $P^b_{t} g=P_{t}(g^{tb})$.

We define the following {\it parabolic martingale}
\begin{equation}
  \label{eq:dpm}
  F^b_t=F^b_t(x;u,f)=P^b_{u-t} f(x+X^b_{t})=P_{u-t} f^{ub}(x+X_{t})=F_t(x;u,f^{ub})  \,,
\end{equation}
where we write $F_t=F^0_t$.
By \cite[Lemma~2]{MR2345912} and (\ref{eq:dpm}),
$t\mapsto F^b_t$ is indeed a (bounded) $\{\F_t\}$-martingale.
In fact, this is very simple because $t\mapsto X_{t}$ is piecewise constant, and so
\begin{eqnarray*}
F_t(x;u,f^{ub})-F_0(x;u,f^{ub})&=&\sum_{S_i\leq t}
[P_{u-v}f^{ub}(x+X_{S_i})-P_{u-v}f^{ub}(x+X_{S_i-})]
\\
&+&\int_0^t\left(\frac{\partial}{\partial v}P_{u-v}\right)f^{ub}(x+X_{v-})dv\,.
\end{eqnarray*}
The equality is a special case of  It\^o formula for the space-time process $t\mapsto (u-t,X_t)$, see, e.g., \cite[Theorem II.31]{MR2020294}, \cite[p. 140]{elliott}.
By (\ref{eq:dpt}) and Lemma~\ref{lem:ls}, the above expression has zero expectation. Furthermore, for $0\leq s\leq t$ we consider
$$
\sum_{s<S_i\leq t}
[P_{u-v}f(x+X_{S_i})-P_{u-v}f(x+X_{S_i-})]
+\int_s^t\left(\frac{\partial}{\partial v}P_{u-v}\right)f(x+X_{v-})dv
\,.
$$
For $v\geq s$ we have $X_v=X_s+(X_{v}-X_s)$, where the two terms are independent, and the process $t\mapsto Y_t=X_{t+s}-X_s$ is compound Poisson.
Integrating against the distribution of $Y$, and using Lemma~\ref{lem:ls} and (\ref{eq:dpt}) we see that $F_t$ is a martingale.

Let $\phi:\Rd\to \C$ and $\|\phi\|_\infty\leq 1$. We define 
$G^b_t=G^b_t(x;u,g,\phi)$ as
\begin{eqnarray}  
\!\!\!\!\!
\!\!\!\!\!
&&
\sum_{ S_i\leq t}
\left[P^b_{u-S_i}
  g(x+X^b_{S_{i}})-P^b_{u-S_i}g(x+X^b_{S_{i}-})\right]\phi(X^b_{S_i}-X^b_{S_i-})\nonumber
\\
\!\!\!\!\!
\!\!\!\!\!
&& - \int_0^t \int_\Rd
\left[P^b_{u-v}g(x+X^b_{v-}+z)-P^b_{u-v}g(x+X^b_{v-})\right]\phi(z)\nu(dz)dv\,.
\label{df:martF}
\end{eqnarray}
We let $G_t=G^0_t$, and note that 
$G^b_t(x;u,g,\phi)=G_t(x;u,g^{ub},\phi)$.
It now follows from \cite[Lemma~3 and Lemma~4]{MR2345912}, or a similar reasoning as above, that
$G^b_t$ is an $\{\F_t\}$-martingale, and $\ex |G_t|^p<\infty$ for every $p>0$. 
We also have 
\begin{eqnarray} \label{l:imf}
F^b_t(x;u,f)&=&F_t(x;u,f^{ub})=G_t(x;u,f^{ub},1)+P_{u}(f^{ub})(x)\,.
\end{eqnarray}
Let $n\to \infty$.
Since $G^b_t$ is square integrable, by orthogonality of increments we have
\begin{eqnarray*}
\ex \abs{G^b_t}^2&=&\ex \sum_{k=1}^n \abs{G^b_{kt/n}-G^b_{(k-1)t/n}}^2\nonumber \\
&\to& 
\ex \sum_{S_i\leq t} 
\abs{P^b_{u-S_i}g(x+X^b_{S_i})-P^b_{u-S_i}g(x+X^b_{S_i-})}^2\abs{\phi(\Delta
X^b_{S_i})}^2\,. \label{eq:qv}
\end{eqnarray*}
The convergence follows from the fact that the integral in
(\ref{df:martF}) is Lipschitz continuous in $t$.
Hence the quadratic variation (\cite{MR2020294}, \cite{MR745449}) of $G^b$ is 
\begin{equation}
[G^b,G^b]_t = \sum_{S_i\leq t} 
\abs{P^b_{u-S_i}g(x+X^b_{S_i})-P^b_{u-S_i}g(x+X^b_{S_i-})}^2\abs{\phi(\Delta
X_{S_i})}^2
\,. \label{eq:qvp}
\end{equation}
By (\ref{l:imf}), the quadratic variation of $F^b$ is 
\begin{equation}\label{eq:qvG}
[F^b,F^b]_t = |P^b_{u}f(x)|^2+\sum_{S_i\leq t} 
\abs{P^b_{u-S_i}f(x+X^b_{S_i})-P^b_{u-S_i}f(x+X^b_{S_i-})}^2
\,. 
\end{equation}
Thus, $G^b(x;u,g,\phi)$ is
{differentially subordinate} to $F^b(x;u,g)$, compare (\ref{eq:dsu}).  
Let $1<p<\infty$.
We may use the result of Wang \cite[Theorem~1]{MR1334160} for general martingales with jumps, to conclude that  
\begin{equation}\label{eq:wmt}
   \ex |G^b_t(x;u,g,\phi)| ^p\leq (p^*-1)^p   \,\ex |F^b_t(x;u,g)|^p\,.
   \end{equation}
Let $t=u$. We have $F^b_u(x;u,f)=f(x+X^b_{u})$. Using (\ref{eq:wmt}) and (\ref{eq:wrtlm}) we obtain
\begin{equation*} 
\int_\Rd \ex  
|G^b_u(x;u,g,\phi)|^p dx 
\leq
(p^*-1)^p   
\int_\Rd \ex |g(x+X^b_{u})|^p dx 
=(p^*-1)^p   
\pl g \pl_p^p\,.
\end{equation*}
By H\"older's  inequality and (\ref{eq:wrtlm}),
\begin{equation}
\label{eq:bof}
\int_\Rd \ex |G^b_u(x;u,g,\phi) f(x+X^b_{u})| dx \leq 
(p^*-1)\pl g \pl_p
\pl f \pl_q\,.
\end{equation}
By (\ref{eq:qv}), (\ref{eq:qvG}) and Lemma~\ref{lem:ls}, 
\begin{eqnarray}
&&\ex G^b_u(x;u,g,\phi) {F^b_u(x;u,f)}   = \ex G^b_u[F^b_u-P^b_{u}f(x)]
\nonumber
\\
&=&\ex \sum_{S_i\leq u} 
\left[P^b_{u-S_i}g(x+X^b_{S_i})-P^b_{u-S_i}g(x+X^b_{S_i-})\right]
\nonumber\\
&&\;\;\;\;\;\;\;\;\;\;\;\;\;[P^b_{u-S_i}f(x+X^b_{S_i})-P^b_{u-S_i}f(x+X^b_{S_i-})]
\phi(\Delta X_{S_i})
\nonumber\\
&=&
\ex \int_0^u \int_\Rd 
\left[P^b_{u-v}g(x+X^b_{v-}+z)-P^b_{u-v}g(x+X^b_{v-})\right]
\nonumber\\
&&\;\;\;\;\;\;\;\;\;\;\;\;\;\;\;[P^b_{u-v}f(x+X^b_{v-}+z)-P^b_{u-v}f(x+X^b_{v-})]
\phi(z)\nu(dz)dv
\nonumber\\
&=& \int_0^u \int_\Rd  \int_\Rd
\left[P^b_{u-v}g(x+y+z)-P^b_{u-v}g(x+y)\right]\label{eq:3245}\\
&&\;\;\;\;\;\;\;\;\;\;\;\;\;\;\;\;[P^b_{u-v}f(x+y+z)-P^b_{u-v}f(x+y)]\phi(z)\nu(dz)p^b_{v}(dy)dv\,.
\nonumber
\end{eqnarray}
To justify applications of Fubini's theorem in what follows, we note that
(\ref{eq:wrtlm}) and the finiteness of $\nu$ imply
\begin{align}
&\int_\Rd \int_0^u \int_\Rd  \int_\Rd
\left|P^b_{u-v}g(x+y+z)-P^b_{u-v}g(x+y)\right|
\nonumber\\
&\;\;\;\;\;\;\;\;\;\;\;\;\;\;\;\;\left|P^b_{u-v}f(x+y+z)-P^b_{u-v}f(x+y)\right|\phi(z)\nu(dz)p^b_{v}(dy)dvdx
\nonumber\\
&\leq 4 \pl g\pl_\infty \|\phi\|_\infty |\nu| \int_0^u 
\pl P^b_{u-v}f\pl_1dv\le   4u\pl g\pl_\infty  \pl f\pl_1|\nu|<\infty
\,.\label{eq:3246}
\end{align}
We consider
$$
\Lambda(g,f)=\int_\Rd \ex G_u(x;u;g,\phi) F_u(x;u,f) dx \ .
$$
Using (\ref{eq:wrtlm}), (\ref{eq:3245}), Plancherel\rq{}s theorem and (\ref{eq:nac}), we see that $\Lambda(g,f)$ equals
\begin{align}
&
\int\limits_0^u
\int\limits_\Rd 
\int\limits_\Rd \left[P^b_{u-v}g(x+z)-P^b_{u-v}g(x)\right]
[P^b_{u-v}f(x+z)-P^b_{u-v}f(x)]\phi(z)\nu(dz)dxdv\,
\nonumber \\
&= 
(2\pi)^{-d}\int\limits_0^u 
\int\limits_\Rd \int\limits_\Rd
|e^{-i(\xi,z)}-1|^2\hat{g}(\xi){\hat{f}(-\xi)}\label{eq:znikabe}e^{2(u-v)\Re \Psi(\xi)}\phi(z)\nu(dz)d\xi dv 
\\
&=
(2\pi)^{-d}\int\limits_\Rd\hat{g}(\xi)\hat{f}(-\xi) M_u(\xi) d\xi\ \,,
\nonumber
\end{align}
where
\begin{equation}
  \label{eq:fss}
M_u(\xi)=\frac{\int_\Rd[\cos(\xi,z)-1]\phi(z)\nu(dz)}
{\int_\Rd[\cos(\xi,z)-1]\nu(dz)}\left[1-e^{2u\Re\Psi(\xi)}\right]\,.
\end{equation}
By (\ref{eq:bof}) we have that $|\Lambda(g,f)|\leq (p^*-1)\pl g \pl_p
\pl f \pl_q$. By the Riesz representation theorem there is a function $h\in \Lt\cap \Lp$ such that $\pl h\pl_p\leq (p^*-1)\pl g\pl_p$, and
$$
\Lambda(g,f)= \int_\Rd h(x){f}(x)dx\\
= (2\pi)^{-d} \int_\Rd \hat h(\xi){\hat f(-\xi)}d\xi\,.
$$
We conclude that the mapping $g\mapsto h$ is a Fourier multiplier with the symbol $M_u$, and its norm norm is at most $p^*-1$ on $\Lp$. If $V$ is an arbitrary L\'evy measure, then we consider $\varepsilon>0$ and define $\nu$ as the restriction of $V$ to
$\{z:\,|z|>\varepsilon\}$. We let $u\to \infty$ and $\varepsilon\to 0$, and use Fatou\rq{}s lemma as in Section~\ref{BMs:p} and after (\ref{eq:lcs}), to obtain the symbol (\ref{eq:fs_bez_mu}), and the bound $p^*-1$ for general L\'evy measures.

We note that the drift vector $b$ and the asymmetry of the L\'evy measure disappear from our formulas in (\ref{eq:znikabe}).

\section{Further discussion and examples}\label{s:d}
We will comment on the relation between (\ref{BMeq:fs}) and (\ref{eq:pc}).
We first remark that the matrices $\mac{A}$, $\mac{B}$ given by (\ref{eq:mx})  
are symmetric.
  We have 
  \begin{equation}\label{eq:mAB}
  \mac{A} \xi =\int_\sphere  \theta\, (\xi,{\theta})\varphi(\theta) \,\mu(\theta)\quad \mbox{ and }\quad \mac{B} \xi =\int_\sphere  \theta\, (\xi,{\theta}) \,\mu(\theta)
\ ,
\quad \xi \in \Rd\ .
\end{equation}
A natural question arises: How to find $\mu$ and $\varphi$ for given symmetric matrices $\mac{A}$ and $\mac{B}$?
We will focus on $\mac{B}=\mac{I}$, the identity matrix.
\begin{lemma}\label{lem:scf}
If $\mac{A}$ is a complex symmetric $d\times d$ matrix, and $|\mac{A}\xi|\leq |\xi|$ for $\xi \in \R^d$,
then a finite measure $\mu\ge 0$ and a function $\varphi$ on  $\sphere$ exist such that $\pl \varphi\pl_\infty\leq
2$, 
$$ \int_{\sphere} \left( \xi ,  \theta \right)^2 \mu\left(d\theta \right)=(\xi,\xi)\quad \mbox{ and } \quad(\mac{A}\xi,\xi)=\int_{\sphere} \left( \xi ,  \theta \right)^2 \varphi\left( \theta \right) \mu\left(d\theta \right)\,, \quad \xi\in \Rd
 \ .
 $$
 If $\Re \A$ and $\Im \A$ commute, then we may select $\|\varphi\|_\infty\leq 1$.
\end{lemma}
\begin{proof}
We emphasize that $\A$ is symmetric but not necessarily Hermitian. Assume first that $\A$  is normal, that is $\Re \A$ and $\Im \A$ commute. Then they have common eigenvectors $a_k\in \Rd$, and so $\A a_k=\lambda_k a_k$, where $\lambda_k\in \C$, and $|\lambda_k|\leq 1$ for $k=1,\ldots,d$.
For $\xi\in \Rd$,
$$\sum_{k=1}^d (\xi,a_k)^2=|\xi|^2\,,$$ 
and
$$(\A \xi, \xi)=\left(\sum_{k=1}^d (\xi,a_k)\A a_k,\sum_{k=1}^d (\xi,a_k)a_k\right)
=\sum_{k=1}^d \lambda_k (\xi,a_k)^2\,.$$
We may now choose $\mu=\sum_{k=1}^d \delta_{a_k}$ and $\varphi(a_k)=\lambda_k$, so that $\|\varphi\|_\infty\leq 1$. Here $\delta_a$ is the Dirac measure at $a$.

If $\Re  \A$ and $i\Im \A$ do not commute then we may consider each of them separately as in the first part of the proof. We may add the respective measures $\mu$, and $\varphi\mu$. We see that the resulting $\varphi$ is bounded by $1$, but we only obtain a representation of $(\mac{A}\xi,\xi)/[2(\xi,\xi)]$. 
This ends the proof.
\end{proof}
For instance, we consider $\mac{A}$ given by (\ref{eq:R2}). Since $\mac{A}$ is real, and
$\mac{A}(\xi_1,\xi_2,\xi_3,\ldots,\xi_n)=(-\xi_2,-\xi_1,0,\dots,0)$, by Lemma~\ref{lem:scf} and Theorem~\ref{th:bfm} we see that the multiplier with the symbol $-2\xi_1\xi_2/|\xi|^2$ is bounded on $\Lp$ for all $p\in (1,\infty)$, and its norm is not greater than $p^*-1$. The operator is the composition $2R_1R_2$, where $R_j$  is a Riesz transform the first order, i.e. the Fourier multiplier with the symbol $i\xi_j/|\xi|$. Here $i=\sqrt{-1}$, $\xi\in \Rd$ and  $j=1,\ldots,n$. 
As noted in the Introduction the norm actually {\it equals} $p^*-1$ (\cite[Corollary 3.2]{MR2551497}).
  
If $|\mac{A}\xi|\leq c|\xi|$ for $\xi \in \R^d$, then Section~\ref{BMs:p} gives the norm bound $c(p^*-1)$ for the multiplier with the symbol $(A\xi,\xi)/|\xi|^2$, whereas Lemma~\ref{lem:scf} in general only gives $2c(p^*-1)$.  This is disconcerting, but in the following important special case the gap dissapears.

We will consider the Beurling-Ahlfors operator. It is the singular integral 
on the complex plane $\C$ (identified with $\Rt$), defined for smooth compactly supported functions $g$ as follows,
\begin{equation}
   \label{eq:ba}
   B g(z)=-\frac{1}{\pi}\, p.v. \int_{\C}
   \frac{g(w)}{(z-w)^2}dm(w)\,, \quad z\in \C \ .
\end{equation} 
Here $m$ is the planar Lebesgue measure. It is well known that $B$ is a Fourier multiplier with the symbol 
\begin{equation}\label{eq:sBA}
M(\xi)=\frac{\overline{\xi}^2}{\abs{\xi}^2}=e^{-2i\arg \xi} \ ,
 \,
\end{equation}
where $\xi=(\xi_1, \xi_2)\in \mathbb{R}^2$ is identified with $\xi_1+i\xi_2\in \C$. For a detailed discussion of $B$, its numerous connections and applications in analysis, partial differential equations and quasiconformal mappings, 
we refer to \cite{2010RB} and the many references given there.  

The above symbol  $M$ is precisely the one given by (\ref{eq:pc}) and (\ref{eq:BAm}). We have $\| A\|=2$, see the Introduction.
Lemma~\ref{lem:scf} and Theorem~\ref{th:bfm} yield
the norm bound $4(p^*-1)$ for $B$ on $\Lp$. However, a detailed inspection shows that $\mu$ uniform on $\{1, i, e^{i\pi/4}, e^{-i\pi/4}\}$, and $\varphi$ such that
$\varphi(1) = 2$, $\varphi(i) = -2$, $\varphi(e^{i\pi/4}) = -2i$, $\varphi(e^{-i\pi/4}) = 2i$, give a more efficient representation (\ref{BMeq:fs}) of (\ref{eq:sBA}), and so $\| B\|\leq 2(p^*-1)$. The bound was first obtained in \cite{NV} by using certain Bellman function constructed from Burkholder's discrete martingale inequalities.  The It\^o calculus approach was presented in \cite{banuelos-mendez} to get the bound, as in our Section~\ref{BMs:p}. The best bound to date for the operator norm of $B$ on $\Lp$ is given in \cite{MR2386238}. We refer the reader to 
\cite{2010RB} for further references, and a thorough discussion of the celebrated conjecture of T. Iwaniec, asserting that $\pl B\pl =p^*-1$.

As it stands, our approach seems to be unable to improve the bound $2(p^*-1)$ for (\ref{eq:sBA}), as indicated by the following fact, which should be compared with (\ref{BMeq:fs}).
\begin{lemma}\label{lem:c2} If $\varphi$ and nonzero $\mu\geq 0$ on $\sphere\subset \Rt$ are such that
\begin{equation} \label{eq:BAnegative}
\int_{\sphere} (\xi,\theta)^{2} \ \varphi(\theta)\mu(d\theta)
=e^{-2i\arg \xi} \int_{\sphere}(\xi,\theta)^{2} \ \mu(d\theta)
 \,, \quad \xi \in \Rt \ ,
\end{equation}
 then $\pl \varphi\pl_\infty  \geq 2$.
\end{lemma}
\begin{proof}
We can assume that $\varphi$ is bounded. 
We denote $t=\arg \xi$, $s=\arg \theta$, and identify $\varphi(\theta)$ and $\mu(d\theta)$ with $\varphi(s)$ and $\mu(ds)$, correspondingly.
We have 
\[
(\xi,\theta)^2=\cos^2 (t-s) 
=\frac{1}{2}\big(\cos (2(t-s))+1\big)=
\frac{1}{2}+\frac{1}{4}e^{2it}e^{-2is}+\frac{1}{4}e^{-2it}e^{2is} \,, 
\]
and hence the left-hand side of (\ref{eq:BAnegative}) is
\[
\frac{1}{2}\int_\sphere \varphi(s)\mu(ds)+\frac{1}{4}e^{2it}\int_\sphere e^{-2is}\varphi(s)\mu(ds)+\frac{1}{4}e^{-2it}\int_\sphere e^{2is}\varphi(s)\mu(ds).
\]
However, the right-hand side equals
\[
\frac{1}{2}e^{-2it}\int_\sphere \mu(ds)+\frac{1}{4}\int_\sphere e^{-2is}\mu(ds)+\frac{1}{4}e^{-4it}\int_\sphere e^{2is}\mu(ds) \ .
\]
In particular,
\[
\frac{1}{4}\int_\sphere e^{2is}\varphi(s)\mu(ds)=
\frac{1}{2}\int_\sphere \mu(ds) \ ,
\]
which is impossible if $\pl \varphi\pl _\infty<2$.
\end{proof}
Let $\mu(ds)=ds$. In view of the above proof, $\varphi(s)=e^{iks}$ with integer $k\neq -2, 0, 2$, yields the zero symbol.
If $\varphi(s)=e^{\pm 2 is}$ then we arrive at $e^{\pm 2i\arg \xi}/2$,
in particular we obtain an elegant representation of (\ref{eq:sBA}).

Let $V$ be the L\'evy measure of a non-zero symmetric $\alpha$-stable L\'evy process in $\Rd$, with $\alpha \in (0,2)$. In polar coordinates we have (see, e.g., \cite{MR1739520}, \cite{MR2320691})
 \begin{equation} \label{eq:spc}
   V(drd\theta) = r^{-1-\alpha}dr\sigma(d\theta) \ , \quad r>0 \ , \ \theta\in \sphere \ ,
 \end{equation}
where the so-called spectral measure $\sigma$ is finite and non-zero on $\sphere$.  Let $\varphi$ be complex-valued on $\sphere$ and such that $ \abs{\phi(\theta)} \leq 1$, $\theta\in \sphere$. Let $\phi(z) = \varphi\left(z/\abs{z}\right)$ for $z\neq 0$, and $c_{\alpha}=\int_0^{\infty} (1-\cos s) s^{-1-\alpha} ds$. By a change of variable,
 \begin{eqnarray} \label{eq:sFs}
  \int_{\Rd} [1-\cos(\xi, z)]\phi(z) V(dz) &=& 
                         \int_{\sphere}\int_0^\infty [1-\cos( \xi, r\theta)]\phi(r\theta) r^{-1-\alpha}dr \ \sigma(d\theta) 
                        \nonumber \\
 &=& c_{\alpha}\int_{\sphere} \abs{(\xi,\theta)}^{\alpha}\varphi(\theta)\sigma(d\theta) \,.
 \end{eqnarray}
Theorem \ref{th:bfm} yields a multiplier bounded in $\Lp$ by $p^*-1$, with the symbol
 \begin{equation}
  M(\xi) = \frac {
                  \int_{\sphere} \abs{(\xi,\theta)}^{\alpha}\varphi(\theta)\sigma(d\theta)
                 }
                 {
                  \int_{\sphere} \abs{(\xi,\theta)}^{\alpha}\sigma(d\theta)
                 } \,.
 \end{equation}
 In particular, for $j=1,\ldots,d$, we obtain 
\begin{equation}
  M(\xi)=\frac{|\xi_j|^\alpha}{|\xi_1|^\alpha+\cdots+|\xi_d|^\alpha}\,,\quad
  \xi=(\xi_1,\ldots,\xi_d)\in \Rd\,.
\end{equation}
These are Marcinkiewicz-type multipliers, as in \cite[p. 110]{MR0290095}.

In the next example we will specialize to $\Rt$. Let $\sigma$ be the Lebesgue measure on the circle, and $\varphi(\theta) = e^{-2i\arg \theta}$, as in the comment following Lemma~\ref{lem:c2}. Let $\xi\in \R^2$ and $t=\arg \xi$. In view of (\ref{eq:sFs}), the numerator of the symbol is
 \begin{eqnarray*}
 c_\alpha\abs{\xi}^{\alpha}\int_0^{2\pi} \abs{\cos (t-s)}^{\alpha}e^{-2is}ds 
&=&  c_\alpha\abs{\xi}^{\alpha}e^{-2it}
\int_0^{2\pi} \abs{\cos v}^{\alpha}e^{2iv}dv \\
 								&=& c_\alpha\abs{\xi}^{\alpha}e^{-2it}\int_0^{2\pi} \abs{\cos v }^{\alpha}\cos(2v)dv \ .
 \end{eqnarray*} 
For $a,b > -1$ we have
 \begin{equation} \nonumber
  \int_{0}^{\frac{\pi}{2}}\sin^av\cos^bv\,dv = \frac{1}{2}\Beta\left(\frac{a+1}{2},\frac{b+1}{2}\right) = \frac{1}{2}\frac{\Gamma\left(\frac{a+1}{2}\right)\Gamma\left(\frac{b+1}{2}\right)}{\Gamma\left(\frac{a+b+2}{2}\right)} \ ,
 \end{equation} 
see, e.g., \cite[Chapter I]{lebedev}. Therefore
 \begin{align*}
&  \int_{0}^{2\pi} \abs{\cos(v)}^{\alpha}\cos(2v)dv = \int_{0}^{2\pi} \abs{\cos(v)}^{\alpha}(2\cos^2 v -1)dv \\
 & = 4\Beta\left(\frac{\alpha+3}{2},\frac12\right) - 2\Beta\left(\frac{\alpha+1}{2},\frac12\right)
  = \frac{2\alpha}{\alpha+2} \Beta\left(\frac{\alpha+1}{2},\frac12\right),
 \end{align*}
  where we used the identity $\Gamma(x+1) = x\Gamma(x)$. Since
 $$
  \int_{0}^{2\pi} \abs{\cos(v)}^{\alpha}dv = 2\Beta\left(\frac{\alpha+1}{2},\frac12\right),
 $$
 we obtain the symbol
\[
  M(\xi) = \frac{\alpha}{\alpha+2} e^{-2i\arg \xi} \ .
\]
For $\alpha\to 2$ we recover the bound $2(p^*-1)$ for the Beurling-Ahlfors transform. 

We will consider more general L\'evy measures in $\Rd$ of product form
in polar coordinates,
 \begin{equation} \label{eq:spcg}
   V(drd\theta) = \rho(dr)\sigma(d\theta) \ , \quad r>0 \ , \ \theta\in \sphere \ .
 \end{equation}
Here $\sigma$ is finite on $\sphere$ and $\int_0^\infty r^2\wedge 1 \  \rho(dr)<\infty$.
An interesting class of such measures are the so-called tempered stable L\'evy processes (\cite{MR2327834}, \cite{MR2591907}). The following example is on the borderline of the tempered stable processes.
Let 
$$\rho(dr)=e^{-r}\frac{dr}{r}\,.$$
In view of the calculations following (\ref{eq:spc})
 we have
$$
\int_0^\infty [1-\cos(\xi,r\eta)]\rho(dr)=\int_0^\infty [1-\cos x]e^{-x/|(\xi,\theta)|}\frac{dx}{x}\,.
$$
The Laplace transform of $(1-\cos x)/x$ equals $0.5 \ln \left(1+s^{-2}\right)$.
Theorem \ref{th:bfm} yields a multiplier bounded in $\Lp$ by $p^*-1$, with the symbol
 \begin{equation}
  M(\xi) = \frac {
                  \int_{\sphere} 
                  \ln \left[1+(\xi,\theta)^{-2}\right]\varphi(\theta)\sigma(d\theta)
                 }
                 {
                  \int_{\sphere} \ln \left[1+(\xi,\theta)^{-2}\right]\sigma(d\theta)
                 } \,,
 \end{equation}
 provided $|\varphi|\leq 1$ on $\sphere$.
For instance, for $j=1,\ldots,d$, we  obtain 
\begin{equation}
  M(\xi)=\frac{\ln \left(1+\xi_j^{-2}\right)}
  {
  \ln \left(1+\xi_1^{-2}\right)+\cdots+
  \ln \left(1+\xi_d^{-2}\right)}\, , \quad \xi \in \Rd \ .
\end{equation}
We conclude with a few general remarks. It is  well known that the stochastic calculus of can be used to obtain non-symmetric Fourier symbols by composing the Brownian motion with harmonic functions, thus by {\it harmonic} rather than parabolic martingales.  This goes back to the pioneering paper of Gundy and Varopoulos \cite{MR545671} for Riesz transform, and we again refer the reader to the survey paper \cite{2010RB} for further discussion.
We also note that McConnell studied in \cite{MR752501} extensions of the 
H\"ormander multiplier theorem. He used the Cauchy process composed with harmonic functions on the upper half-space in $\R^{d+1}$. This may be considered a special case of our parabolic martingales, see \cite[Lemma~2.1]{MR752501}. However, the Cauchy process is obtained by optional stopping of the $(d+1)$-dimensional Brownian motion on the half-space, and so \cite{MR752501} is more related to the work of Gundy and Varopoulos \cite{MR545671} than to the parabolic martingales of Ba\~nuelos and M\'endez-Hernand\'ez \cite{banuelos-mendez}.

It is interesting if the bound for $\varphi$ in the conclusion of Lemma~\ref{lem:scf} may be improved for general complex symmetric matrices $\A$. 
In this connection we also note that if $\mu\geq$ and $\|\phi\|_\infty\leq 1$ (on $\sphere$), then by (\ref{eq:mAB}),
$$
(\mac{A}\xi,\xi)=\int_\sphere (\xi,\theta)^2\varphi(\theta)\mu(d\theta) \quad \mbox{and} \quad
(\mac{B}\xi,\xi)=\int_\sphere (\xi,\theta)^2\mu(d\theta)\ , \quad \xi \in \Rd\ ,
$$
We thus see that $\mac{B}$  is nonnegative definite, and 
\begin{equation}\label{eq:cAB}
|(\mac{A}\xi,\xi)|\leq (\mac{B}\xi,\xi)
 \ , \quad \xi \in \Rd \,. 
\end{equation} 
Of course, (\ref{sub1}) implies (\ref{eq:cAB}) for $\mac{B}=\mac{I}$, but
the relationship between (\ref{eq:cAB}) and the conclusion of Lemma~\ref{lem:scf} calls for further study.

If L\'evy measures satisfy $\nu_1\leq \nu_2$, then 
\begin{equation} \label{eq:fspp}
M\left(\xi\right)=
\frac{
      \int_\Rd \left[1 - \cos (\xi  ,  z) \right]\nu_1(dz)  
          }
     {
           \int_\Rd \left[1 - \cos(\xi  ,  z) \right]\nu_2(dz) 
     } \ ,
\end{equation} 
defines an $\Lp$ multiplier with the norm not exceeding $p^*-1$. This follows from Theorem~\ref{th:bfm} with $V=\nu_2$, $\phi=d\nu_1/d\nu_2$ and $\mu=0$.
The observation allows to study inclusions between anisotropic Sobolev spaces (\cite{MR2245255}).

Surprisingly, non-symmetric L\'evy processes did not bring about non-symmetric symbols here.
We owe to Mateusz Kwa\'snicki yet another explanation of this phenomenon, using time reversal of L\'evy processes (private communication). 
Our present discussion leaves wide open the problem of modifying the jumps of L\'evy processes in such a way as to obtain non-symmetric multipliers. 

An interesting problem, indirectly touched upon by Lemma~\ref{lem:scf}, is the following: 
Can we handle a class of Fourier multipliers on $\Lp$ by specifying the denominator and some boundedness and differentiability properties of the ratio (\ref{eq:fs}), so to recover bounded $\phi$ and  $\varphi$ from these?
\bigskip

\noindent
{\bf Acknowledgement:} {We thank Mateusz Kwa\'snicki  for discussions on symmetrization and Stanis{\l}aw Kwapie\'n for  comments on \cite{MR752501} and remarks on Section~\ref{BMs:p}, which helped to simplify our proofs.  This paper was partially completed during a visit by the first author to Wroc{\l}aw University of Technology and the Mathematical Research and Conference Center in Bedlewo, supported by grant MNiSW N N201 397137.  He gratefully acknowledges these institutions for their warm hospitality.}

\bibliographystyle{abbrv}
\bibliography{martingale_transforms}


\end{document}